\documentclass{article}
\usepackage{amsmath}
\usepackage{amsfonts}
\usepackage{amssymb}
\usepackage{amsthm}
\usepackage{graphicx}
\usepackage{tikz}
\usepackage{cleveref}

\makeindex

\def\BB{\mathbb B}

\def\RR{\mathbb R}

\def\ZZ{\mathbb Z}

\def\UU{\mathbb U}

\def\YY{\mathbb Y}

\def\bx{\mathbf{x}}
\def\bc{\mathbf{c}}

\def\by{\mathbf{y}}
\def\bz{\mathbf{z}}

\def\cH{\mathcal{H}}
\def\cF{\mathcal{F}}
\def\cG{\mathcal{G}}

\def\cR{\mathcal{R}}

\def\cC{\mathcal{C}}

\def\cN{\mathcal{N}}

\def\cT{\mathcal{T}}
\def\cA{\mathcal{A}}

\def\cF{\mathcal{F}}
\def\cP{\mathcal{P}}
\def\cQ{\mathcal{Q}}
\def\cR{\mathcal{R}}

\newcommand{\ga}{\alpha}
\newcommand{\gb}{\beta}
\newcommand{\gc}{\gamma}
\newcommand{\gd}{\delta}
\newcommand{\gr}{\rho}

\newcommand{\go}{\omega}
\newcommand{\gk}{\chi}

\newtheorem{theorem}{Theorem}

\newtheorem{corollary}{Corollary}

\newtheorem{definition}{Definition}

\newtheorem{lemma}[theorem]{Lemma}

\begin{document}
	
	\title{
		A Dynamic Programming Algorithm for Binary Polynomial Optimization
	}

	\author{
		Endre Boros\thanks{MSIS Department and RUTCOR, Rutgers University, New Brunswick, New Jersey, USA. Email: \texttt{endre.boros@rutgers.edu}.}
	}
	
	\date{\today}
	
	\maketitle
	
\begin{abstract}
	In a binary polynomial optimization problem (BPO, in short) we are maximizing a  multilinear polynomial expression depending on $n$ binary variables. This is a hard optimization class, containing many NP-hard problems, including unconstrained quadratic binary optimization. Several tractable special classes were considered in the literature, in particular when the input is a quadratic polynomial.
		
	We focus on a general variable elimination scheme for BPO, 
	and develop the unique explicit multilinear polynomial form for the equivalent BPO problem obtained after the elimination of a subset of the variables. 
	The obtained closed form representation of such an equivalent BPO problem allows us to characterize special classes for which this elimination method, when applied recursively, provides a computationally efficient solution for BPO.  Our approach is elementary\footnote{We do not use LP models and algorithms; this is a purely algebraic dynamic programming approach.}, algebraic, and provides efficient solution to a wide problem class that properly generalizes several of the known tractable special cases.
\end{abstract}	

\section{Introduction}\label{s-intro}

Unconstrained binary polynomial optimization (or in short BPO) aims at finding the maximum value of a polynomial expression depending on $n$ binary variables. We use $\BB=\{0,1\}$ to denote the binary values, $V=[n]=\{1,2,\dots,n\}$ to denote the index set of the variables, and use $x_v\in \BB$ for $v\in V$ for the binary variables in our problem. We denote by $\bx=(x_v\mid v\in V)$ a vector of 
these binary variables, by $\BB^V$ the set of binary vectors, by $\RR$ the set of real numbers, by $\ZZ$ the set of integers, and by $\ZZ_+$ the set of nonnegative integers. 

To a given hypergraph $\cH\subseteq 2^V$ (that is a collection of subsets of $V$) and integer vector $\bc=(c_H\mid H\in \cH)\in \ZZ^{\cH}$ we associate a multilinear polynomial function 
\begin{equation}\label{e-multlinpoly}
f_{\cH,\bc}(\bx) ~=~ \sum_{H\in \cH} c_H \prod_{v\in H} x_v. 
\end{equation}
With this notation, we can formulate the \textsc{BPO problem} as 
\begin{equation}\label{e-BPO}
	\max_{\bx\in\BB^V} f_{\cH,\bc}(\bx).
\end{equation}
Note that for a binary value $x\in \BB$ we have $x^k=x$ for all integers $k\geq 1$, and thus any polynomial expression in these variables is equivalent to a multilinear one. Note also that $\cH$ is not assumed to be a Sperner hypergraph, and in particular, it may contain a hyperedge of size $0$ corresponding to a constant term, and hyperedges of size $1$ corresponding to the linear terms of $f_{\cH,\bc}$, etc. 
In discrete optimization it is common to assume rational coefficients. Multiplying our objective function with the common denominator of the coefficients of a given rational input function is not changing the input size in a non-polynomial way. So, for simplicity, we assume in the sequel that our input coefficients are integers\footnote{For the sole purpose of being able to make more precise complexity claims about our procedures.}. 
Let us introduce $Z(f_{\cH,\bc})=\sum_{H\in \cH} |c_H|$ and note that the (binary) input size of problem \eqref{e-BPO} is bounded by $O(|\cH||V|\log Z(f_{\cH,\bc}))$. 

A mapping $g:\BB^V\mapsto \RR$ is called a \emph{pseudo-Boolean function}. 
It is well-known that every pseudo-Boolean function has a unique multilinear polynomial representation $f_{\cH,\bc}$, where $\cH\subseteq 2^V$ and $\bc\in \RR^\cH$, see \cite{Hammer-Rosenberg-Rudeanu-1963,Hammer-Rudeanu-1968}.
Problem BPO is a special case of pseudo Boolean optimization\footnote{While a pseudo-Boolean function has a unique multilinear polynomial form, it also has several other representations, so called posiforms, some of which may arise naturally from the application we consider, and some of which may be exponentially smaller in size than the unique multilinear polynomial, see \cite{Boros-Hammer-2002} for more details.}, and it contains quadratic unconstrained binary optimization (QUBO) as a special case. Thus, BPO is an NP-hard optimization class, since QUBO is known to include several well-known NP-complete combinatorial optimization problems, see e.g. \cite{Boros-Hammer-2002,KHGLLWW-2014,punnen2022quadratic} for a more comprehensive list of related problems, equivalences, and applications.  

The structure of the hypergraph $\cH$ plays an important role in the hardness/easiness of BPO. The so called basic-algorithm introduced by \cite{Hammer-Rosenberg-Rudeanu-1963-b,Hammer-Rosenberg-Rudeanu-1963} for pseudo-Boolean optimization was shown to be polynomial for BPO problems when $\cH$ has bounded tree-width \cite{Crama-Hansen-Jaumard-1990}. The standard linearization technique was shown to lead to an integral polytope, and hence to polynomial time solvability of BPO when $\cH$ is Berge-acycic \cite{Buchheim-Crama-Heck-2019,DelPia-Khajavirad-2018a}. Unlike for graphs, acyclicity is not a uniquely definable notion for hypergraphs, see e.g., \cite{Brault-Baron-2017,Duris-2012,Fagin-1983b,Jegou-Ndiaye-2009,Ordyniak-Paulusma-Szeider-2013}. 
For the notion of the so-called $\beta$-acyclicity, it was shown in \cite{DelPia-Khajavirad-2021} that BPO is polynomially solvable if in addition $\cH$ does not contain a special substructure, called a kite. For the family of so called cycle-hypergraphs \cite{DelPia-DiGregorio-2021INFORMS} proved that BPO is polynomially solvable. More recently \cite{DelPia-DiGregorio-2022,DelPia-DiGregorio-2023} showed that BPO is solvable in strongly polynomial time whenever $\cH$ is $\beta$-acyclic (without any additional assumptions). 
The approach of  \cite{DelPia-DiGregorio-2022,DelPia-DiGregorio-2023} is based on the iterative elimination of the \emph{nest points}, the existence of which is guaranteed by $\beta$-acyclicity. This idea was generalized in \cite{Lanzinger-2023} by considering \emph{nest sets}. Recently \cite{Clausen-Crama-Lusby-Heck-Ropke-2024} introduced a new dynamic programming based approach that was shown to work efficiently for hypergraphs with \emph{limited reach} which is a subfamily of hypergraphs with bounded tree-width. Utilizing an enhanced linearization model polynomial solvability of BPO problems was shown in \cite{DelPia-Khajavirad-2025} for $\ga$-acyclic hypergraphs with bounded rank and in \cite{Capelli2026} for hypergraphs with bounded incidence tree-width. 
While supermodular functions can be maximized in polynomial time \cite{Schrijver2000combinatorial}, recognizing supermodular multilinar polynomials is co-NP-hard for degree $4$ or higher \cite{Gallo-Simeone-1989supermodular}. Quadratic or cubic supermodular BPO problems can be solved in polynomial time using network flow based methods, see \cite{Billionnet-Minoux-1985maximizing,Hammer1965}. There are numerous special classes of quadratic BPO problems that are solvable in polynomial time by polyhedral, linear programming based techniques, see \cite{Boros-Hammer-2002,KHGLLWW-2014,punnen2022quadratic}. 
(We are going to introduce all necessary notions and definitions later, in our technical sections.)

\subsection{Our Contributions}

In this paper we focus on a dynamic programming variable elimination method, a generalized version of the so called \emph{basic-algorithm} for problem BPO, and provide a complete analysis and description of this approach. 

The basic-algorithm, as introduced in \cite{Hammer-Rosenberg-Rudeanu-1963-b,Hammer-Rosenberg-Rudeanu-1963},
is a rather generic variable elimination scheme, described in mathematical terms. The important missing algorithmic details e.g., for bounded tree-width hypergraphs were added and analyzed by \cite{Crama-Hansen-Jaumard-1990}. In fact, several of the above cited positive results can be viewed as variants of this basic-algorithm, filing in the missing algorithmic details and complexity analysis for particular hypergraph structures. The basic-algorithm eliminates variables one-by-one. We consider a more general scheme, in which we eliminate subsets of variables.

For a subset $S\subseteq V$ and binary vector $\bx\in\BB^V$, we define $\bx^S = (x_j\mid j\in S)$ the subvector of $\bx$ corresponding to index set $S$. For a partition $V=X\cup Y$ we also write $\bx=(\bx^X,\bx^Y)$.

Assume $\cH\subseteq 2^V$ and $\bc\in \ZZ^\cH$, as in \eqref{e-BPO}. For a nontrivial partition $V=X\cup Y$ (with $X\neq\emptyset$, $Y\neq\emptyset$), we define the \emph{projection} of $f_{\cH,\bc}$ on $Y$ as
\begin{equation}\label{e-eliminateX}
	f_{\cH,\bc}^{Y}(\bx^Y) ~=~ \max_{\bx^X\in \BB^X} f_{\cH,\bc}(\bx^X,\bx^Y) \text{ for all } \bx^Y\in\BB^Y.
\end{equation}
The results of \cite{Hammer-Rudeanu-1968} show that if $g:\BB^V\mapsto \ZZ$ is an integer mapping, then it's unique multilinear polynomial has integer coefficients. Thus, it follows that the unique multilinear polynomial form of $f_{\cH,\bc}^{Y}$ also has integer coefficients for every $Y\subseteq V$. 

In our first result we provide an explicit closed form for the unique multilinear polynomial of $f_{\cH,\bc}^{Y}$ to which we arrive by utilizing a higher dimensional representation of the problem, the so called signature space of a hypergraph (considered e.g., in \cite{Berczi-2021,Capelli-Florent-Strozecki-Yann-2021}.) We prove that in this unique multilinear polynomial the monomials that (may) have nonzero coefficients are in a one-to-one relation with the signatures of $\cT=\{H\setminus X\mid H\in\cH,~ H\cap X\neq\emptyset\}$. This family of monomials depends only on the hypergraph $\cH$ and the chosen subset $X$, but do not depend on the coefficients $\bc$ of our input. We also show that this set of monomials and the corresponding coefficients in the unique multlinear polynomial of $f_{\cH,\bc}^{Y}$ can be computed in polynomial time in $|V|$, $|\cH|$, $2^{|X|}$ and the size $K$ of the signature space of $\cT$ (which on its own turn is known to be computable polynomially in terms of $K$, $|V|$, and $|\cH|$, see \cite{Berczi-2021}.)

Our next result is the introduction of a special class of hypergraphs for which the unique multilinear polynomial of the projection $f_{\cH,\bc}^{Y}$ is efficiently computable in terms of the input size, and for which the variable elimination steps can be executed recursively and efficiently, leading to polynomial time solution of BPO. This class provides a common generalization for several of the above cited tractable cases e.g., hypergraphs with limited reach \cite{Clausen-Crama-Lusby-Heck-Ropke-2024} or more generally, for bounded tree-width \cite{Crama-Hansen-Jaumard-1990}, $\beta$-acyclic hypergraphs \cite{DelPia-DiGregorio-2022,DelPia-DiGregorio-2023}, bounded rank $\ga$-acyclic hypergraphs \cite{DelPia-Khajavirad-2025}, and includes some additional cases (e.g., circular interval hypergraphs). 

\bigskip

\section{Signatures of hypergraphs}

Given a hypergraph $\cH\subseteq 2^V$ and $\bx\in\BB^V$, we introduce
\begin{equation}\label{e-y(x)}
y_H(\bx)=\prod_{v\in H} x_v ~~~\text{ and }~~~ \by(\bx)=(y_H(\bx)\mid H\in\cH)\in \BB^\cH.
\end{equation}
We call $\by(\bx)$ a \emph{signature} of $\cH$, and denote by 
\begin{equation}\label{e-YY}
\YY=\{\by(\bx)\mid \bx\in \BB^V\} \subseteq \BB^\cH
\end{equation}
the set of signatures of $\cH$. If the hypergraph is not obvious from the context, we write $\by^\cH(\bx)$ and $\YY_\cH$, respectively.

Signatures were considered in a more general setting in \cite{Berczi-2021} and the following result, applicable to our case, was shown there:

\begin{theorem}[\cite{Berczi-2021}]\label{t-signature-generation}
	Given a hypergraph $\cH\subseteq 2^V$ and a positive integer $K$, we can generate $\YY_\cH$ and show that $|\YY_\cH|<K$ or 
	generate $K$ distinct elements of $\YY_\cH$ in $K\cdot O(|V||\cH|)$ time. \qed
\end{theorem}

\subsection{Lifting multilinear polynomials}\label{subsec-lifting}

Let us now fix a hypergraph $\cH\subseteq 2^V$, an integer vector $\bc\in\ZZ^\cH$, and a nontrivial partition $V=X\cup Y$. We associate to $\cH$, $X$, and $Y$ the following hypergraphs:
\begin{equation}\label{e-PQRT}
\begin{array}{rl}
	\cP &=\{H\in\cH\mid H\subseteq X\},\\
	\cQ &=\{H\in\cH\mid H\cap X\neq\emptyset \text{ and } H\cap Y\neq\emptyset\},\\
	\cR &=\{H\in\cH\mid H\subseteq Y\}, \text{ and }\\
	\cT &=\{H\cap Y\mid H\in \cQ\}.
\end{array}
\end{equation}
We also consider the signatures $\YY_\cT$ of hypergraph $\cT$ and a partition $f_{\cH,\bc}(\bx)=\ga(\bx)+\gb(\bx)+\gc(\bx)$ where 
\[
\begin{array}{rl}
	\ga(\bx) &= \displaystyle\sum_{H\in \cP} c_H \prod_{v\in H} x_v,\\*[5mm]
	\gb(\bx) &= \displaystyle\sum_{H\in \cQ} c_H \prod_{v\in H} x_v, \text{ and }\\*[5mm]
	\gc(\bx) &= \displaystyle\sum_{H\in \cR} c_H \prod_{v\in H} x_v.\\
\end{array}
\]
Note that $\ga$ depends only on variables $x_v$, $v\in X$, while $\gc$ depends only on variables $x_v$, $v\in Y$. We can transform $\gb$ as follows:
\[
\begin{array}{rl}
	\gb(\bx) &= \displaystyle\sum_{H\in \cQ} c_H \prod_{v\in H} x_v,\\*[8mm]
			 &= \displaystyle\sum_{H\in \cQ} c_H \left(\prod_{v\in H\cap X} x_v\right)\left(\prod_{v\in H\cap Y} x_v\right)\\*[8mm]
			 &= \displaystyle\sum_{H\in \cQ} c_H \left(\prod_{v\in H\cap X} x_v\right) y_{H\cap Y}(\bx).
\end{array}
\]
Note that on the right hand side we have sets $H\cap Y$ for $H\in \cQ$ that is sets from $\cT$. Thus, for a particular vector $\by=(y_T\mid T\in \cT)\in \YY_\cT$ we can introduce
\[
\gb_\by(\bx) ~=~ \displaystyle\sum_{H\in \cQ} c_H \left(\prod_{v\in H\cap X} x_v\right) y_{H\cap Y}.
\]
Now, both $\ga$ and $\gb_\by$ depends only on variables from $X$, for each $\by\in \YY_\cT$. Let us finally define
\begin{equation}\label{e-mu}
\mu(\by) ~=~ \max_{\bx\in \BB^X} \left(\ga(\bx)+\gb_\by(\bx)\right) ~~~\text{ for all } \by\in \YY_\cT.
\end{equation}
With these notations we can rewrite the reduced function $f^Y_{\cH,\bc}$ defined in \eqref{e-eliminateX}.

\begin{lemma}\label{l-first-rewrite}
For $\bx\in \BB^Y$ we have
\[
f^Y_{\cH,\bc}(\bx) 
~=~\displaystyle \gc(\bx) +\sum_{\bz\in \YY_\cT} \mu(\bz) \left(\prod_{T\in\cT\atop z_T=1}y_T(\bx)\right)\prod_{T\in\cT\atop z_T=0}(1-y_T(\bx)).
\]
\end{lemma}
\begin{proof}
	Note that the term with coefficient $\mu(\bz)$ is either zero or equals to $\mu(\bz)$, and it is $\mu(\bz)$ exactly for those $\bx\in \BB^Y$ vectors for which we have $\by(\bx)=\bz$.
	Thus the claim follows by \eqref{e-eliminateX}, \eqref{e-y(x)}, and the definition \eqref{e-mu} of $\mu$. 
\end{proof}

Let us remark here that the above expression may not seem to provide a ``small'' multilinear form for the projection $f^Y_{\cH,\bc}$, since to obtain that we would need to carry out the multiplications on the right, for every $\bz\in \YY_\cT$, and for a $\bz$ vector with many zeros that multiplication may yield potentially exponentially many different monomials in terms of the $x_v$, $v\in V$ binary variables. In the next claim we show that in fact this is not happening, and the size is limited in terms of the size of $\YY_\cT$. 

\begin{theorem}\label{l-effective transformation}
	Assume that $\bz^*\in \YY_\cT$. Then there exist unique integer numbers $\gd(\bz^*,\bz)$ for all $\bz\in \YY_\cT$ such that the equality
	\begin{equation}\label{e-unique-coefficients}
	\left(\prod_{T\in\cT\atop z_T=1}y_T(\bx)\right)\prod_{T\in\cT\atop z_T=0}(1-y_T(\bx)) 
	~=~
	\sum_{\bz\in \YY_\cT\atop \bz\geq \bz^*} \gd(\bz^*,\bz) \left(\prod_{v\in \displaystyle\bigcup_{T\in\cT\atop z_T=1}T}x_v\right)
	\end{equation}
	holds for all $\bx\in\BB^Y$. Furthermore, for all $\bz\in \YY_\cT$ we have $|\gd(\bz^*,\bz)|\leq |\YY_\cT|^{|V|}$. 
\end{theorem}

\begin{proof}
	Note first that for all $\bz\in\YY_\cT$ by the definition \eqref{e-y(x)} we have 
	\[
	\left(\prod_{v\in \displaystyle\bigcup_{T\in\cT\atop z_T=1}T}x_v\right)
	~=~ \prod_{T\in \cT\atop z_T=1} y_T(\bx).
	\]
	Thus it is enough to show that for all $\bx\in\BB^Y$ we have the equality
	\begin{equation}\label{e-unique--coefficients-enough}
	\left(\prod_{T\in\cT\atop z^*_T=1}y_T(\bx)\right) \prod_{T\in\cT\atop z^*_T=0} \left(1-y_T(\bx)\right) 
	~=~
	\sum_{\bz\in \YY_\cT\atop \bz\geq \bz^*} \gd(\bz^*,\bz) 
	\prod_{T\in \cT\atop z_T=1} y_T(\bx)
	\end{equation}
	for some uniquely defined integers $\gd(\bz^*,\bz)$ for $\bz\in \YY_\cT$. 
	
	For $\bx\in \BB^Y$ let us define $vol(\by(\bx),\bz^*)$ as the number of vectors $\by\in \YY_\cT$ for which $\by(\bx)\geq \by \geq \bz^*$. 
	
	We are going to prove \eqref{e-unique--coefficients-enough} by induction on $vol(\by(\bx),\bz^*)$. 
		
	Let us observe that if for a binary vector $\bx\in \BB^Y$ we have $vol(\by(\bx),\bz^*)=0$, i.e., 
	$\by(\bx)\not\geq \bz^*$, then we have a set $T\in\cT$ such that $z^*_T=1$ and $y_T(\bx)=0$. Consequently, we get zero on both sides of the above equality (because $y_T(\bx)$ participates in all products.)
	
	Assume next that $\bx\in \BB^Y$ satisfies $vol(\by(\bx),\bz^*)=1$, i.e., 
	$\by(\bx)=\bz^*$. Then we get $1=\gd(\bz^*,\bz^*)$, defining $\gd(\bz^*,\bz^*)$. This is because for all $\bz\in \YY_\cT$, $\bz\geq \bz^*$, $\bz\neq \bz^*$ we have a $T\in \cT$ with $z_T=1$ and $z^*_T=y_T(\bx)=0$ that cancels out the corresponding product on the right hand side.
	
	Let us finally consider the binary vectors $\bx\in \BB^Y$ for which we have 
	$vol(\by(\bx),\bz^*)>1$, i.e., $\by(\bx)\geq \bz^*$, $\by(\bx)\neq \bz^*$. 
		
	We claim first that for all $\bx\in \BB^Y$ for which we have $\by(\bx)\geq \bz^*$, $vol(\by(\bx),\bz^*)>1$
	equation \eqref{e-unique--coefficients-enough} reduces to 
	\begin{equation}\label{e-y>=z^*}
		0 ~=~ \sum_{\bz\in \YY_\cT\atop \by(\bx)\geq \bz\geq \bz^*} \gd(\bz^*,\bz) 
		\prod_{T\in \cT\atop z_T=1} y_T(\bx).
	\end{equation}
	This is because for $\bz\in\YY_\cT$, $\bz\geq\bz^*$, $\bz\not\leq \by(\bx)$ there exists a $T\in \cT$ such that $z_T=1$ and $y_T(\bx)=0$.
	
	Assume that we proved \eqref{e-unique--coefficients-enough} for all $\by(\bx)$ with $vol(\by(\bx),\bz^*)\leq k$, and consider next $\by(\bx)$ with $vol(\by(\bx),\bz^*)=k+1$. If no such $\bx$ exists, then we can consider \eqref{e-unique--coefficients-enough} proven for $k+1$, too. Otherwise from \eqref{e-y>=z^*} we get
	\[
	0~=~\sum_{\bz\in \YY_\cT\atop \by(\bx)\geq \bz \geq \bz^*} \gd(\bz^*,\bz)
	\]
	from which we get
	\[
	\gd(\bz^*,\by(\bx)) ~=~ -\sum_{{\bz\in \YY_\cT\atop \by(\bx)\geq \bz \geq \bz^*}\atop \bz\neq \by(\bx)} \gd(\bz^*,\bz)
	\]
	Since on the right hand side for all participating $\bz$ vectors we have $vol(\bz,\bz^*)\leq k$, the corresponding coefficients $\gd(\bz^*,\bz)$ are already proven to have a unique integer values. Thus the above equality defines uniquely $\gd(\bz^*,\by(\bx))$, as well.
	This completes the induction, and hence completes the proof of our statement.
	
	Note that each new value in this recursion is computed by a simple sum. Furthermore, the depth of the induction is at most $|V|$, thus the claim on the size of the coefficients follows.
\end{proof}

Note that our proof above shows that the right hand side of \eqref{e-unique-coefficients} is the unique multilinear polynomial of the $\BB^{\YY_\cT}\mapsto \ZZ$  mapping defined by the left hand side. (Recall, that all such mappings have a unique multilinear polynomial representation by \cite{Hammer-Rosenberg-Rudeanu-1963,Hammer-Rudeanu-1968}.)

Furthermore, \Cref{l-effective transformation} implies that considering the signature space of $\cT$ is the right choice, in the sense that the ``new'' monomials that may have nonzero coefficient in $f^Y_{\cH,\bc}$ for some $\bc\in \ZZ^\cH$ are exactly the ones on the right hand side of \eqref{e-unique-coefficients} corresponding to the signatures in $\YY_\cT$. 

\bigskip

\begin{theorem}\label{t-main-X-reduction}
	For a hypergraph $\cH\subseteq 2^V$ and nontrivial partition $V=X\cup Y$ we have 
	\begin{equation}\label{e-main-theorem}
	f^Y_{\cH,\bc}(\bx) ~=~ \gc(\bx) 
	+
	\sum_{\by\in\YY_\cT} 
	\left( \sum_{\by^*\in \YY_\cT\atop \by\geq \by^*} \mu(\by^*)\gd(\by^*,\by) \right) \left(\prod_{v\in \displaystyle\bigcup_{T\in\cT\atop y_T=1}T}x_v\right)
	\end{equation}
	and we can compute this in $O(|\cH||\YY_\cT|\max\{|V|,2^{|X|}|X|,|\YY_\cT|\})$ time.
	Furthermore $Z(f^Y_{\cH,\bc})$ is limited by $O(Z(\cH,\bc)|\YY_\cT||V|^{|V|+1})$.
\end{theorem}

\begin{proof}
	We can arrive to this form by substituting the right hand side from \Cref{l-effective transformation} into the form in \Cref{l-first-rewrite}.
	
	We can carry out the computations, in the following way: 
	\begin{itemize}
		\item First we compute the set $\YY_\cT$ in $O(|\YY_\cT||V||\cH|)$ time by \Cref{t-signature-generation}. 	
		\item Next we can determine the $\mu(\by^*)$ values by \eqref{e-mu}. Note that $\ga(\bx)+\gb_\by(\bx)$ has at most $|\cP|+|\cQ|\leq |\cH|$ many terms, each of degree at most $|X|$. Thus computing $\mu(\by)$ for a particular $\by$ vector may take $O(2^{|X|}|X||\cH|)$ time. Doing this for all $\by^*\in\YY_\cT$ thus takes $O(2^{|X|}|X||\cH||\YY_\cT|)$ time. 
		\item Then we pre-compute the $\by\geq \by^*$ pairs and the $vol(\by,\by^*)$ quantities for all $\by^*,\by\in\YY_\cT$ in $O(|\YY_\cT|^2|\cT|)$ time. 
		\item Determining the $\gd(\by^*,\by)$ values, following the inductive procedure in the proof of \Cref{l-effective transformation} may take also $O(|\YY_\cT|^2|\cT|)$ time.
		\item Finally, computing the coefficients for each of the $|\YY_\cT|$ many terms in \eqref{e-main-theorem} can be done in $O(|\YY_\cT|^2)$ time.
	\end{itemize}
	Since $|\cT|\leq |\cH|$, we get in total
	\begin{multline*}
	O(\max\{ |\YY_\cT||V||\cH|, 2^{|X|}|X||\cH||\YY_\cT|,|\YY_\cT|^2|\cH|\}) \\ ~=~ O(|\cH||\YY_\cT|\max\{|V|,2^{|X|}|X|,|\YY_\cT|\}).
	\end{multline*}
	The claim for the size of the coefficients follows by Theorem \ref{l-effective transformation}, since the $|\YY_\cT|\leq 2^{|V|}$ many $\gd(\by^*,\by)$ values are computed by an additive recursion of depth $|V|$ starting with $1$, $\mu(\by^*)$ is limited by $Z(\cH,\bc)$, and we have $|\YY_\cT|$ many ``new'' coefficients. 
\end{proof}

\subsection{Algorithmic Consequences}

Using the above results, we can easily build an algorithm that solves problem \eqref{e-BPO} by iteratively eliminating small (with an eye on the efficiently computable special cases) sets of variables. 

To be able to do this, we need to refine our notation (to avoid potential ambiguities). For a mapping $g:\BB^V\mapsto \ZZ$ let us denote by $\cH(g)\subseteq 2^V$ the hypergraph for which we have $g=f_{\cH(g),\bc}$ for some $\bc \in (\ZZ\setminus\{0\})^{\cH(g)}$ (that is with no zero components). Note that by \cite{Hammer-Rosenberg-Rudeanu-1963,Hammer-Rudeanu-1968} $g$ has a unique multilinear polynomial representation, and thus both $\cH(g)$ and $\bc$ are uniquely defined. 

Let us recall next that to a hypergraph $\cH$ and a partition $V=X\cup (V\setminus X)$ we associated partial hypergraphs as in \eqref{e-PQRT}. In a multistep procedure this may lead to ambiguities, so for clarity we will write  

\begin{equation}\label{e-PQRT-updated}
	\begin{array}{rl}
		\cP = \cP(\cH,X) &=~ \{H\in\cH \mid H\subseteq X\}  \\
		\cQ = \cQ(\cH,X) &=~ \{H\in\cH \mid H\cap X\neq \emptyset \text{ and } H\setminus X\neq \emptyset \} \\
		\cR = \cR(\cH,X) &=~ \{H\in\cH \mid H\cap X = \emptyset \} \\
		\cT= \cT(\cH,X) &=~ \{ H\setminus X \mid H\in \cQ(\cH,X)\}
	\end{array}
\end{equation}

For a hypergraph $\cH\subseteq 2^V$ we define its \emph{union closure} as 
\[
\cH^u ~=~ \left\{\left. \bigcup_{H\in\cH'}H ~\right|~ \cH'\subseteq \cH \right\}.
\]

\begin{lemma}\label{l-union-signature}
	For every $U\in\cH^u$ and binary vector $\by\in \BB^\cH$ defined by 
	\[
	y_H=
	\begin{cases}
		1 & \text{if } H\subseteq U, \text{ and }\\
		0 & \text{otherwise}
	\end{cases}
	~~~~~~\text{ for all }~~~ H\in\cH,
	\]
	we have $\by\in \YY_\cH$. Conversely, for every $\by\in \YY_\cH$ we have
	\[
	\bigcup_{H\in\cH\atop y_H=1} H ~\in~ \cH^u.
	\]
	Consequently, we have $|\cH^u|=|\YY_\cH|$. 
\end{lemma}
\begin{proof}
	It is immediate by the definitions of signatures and union-closures. 
\end{proof}
Thus, the union closure of a hypergraph and its signature space are in a one-to-one correspondence. While in the proofs of the previous section it was advantageous to consider the signature space, in the sequel we switch to the equivalent union closure. 

Let us observe finally that according to \Cref{t-main-X-reduction} we have for all $X\subseteq V$ the containment
\begin{equation}\label{e-th4-consequence}
	\cH\big(f^{V\setminus X}_{\cH,\bc}\big) \subseteq \cR(\cH,X)\cup \cT(\cH,X)^u.
\end{equation}
In fact, for some $\bc\in \ZZ^\cH$ we may have equality here\footnote{Well, given the relative unimportance of attaining equality, it seems surprisingly complicated to settle it rigorously.}. Thus, we feel justified to call the hypergraph $\cR(\cH,X)\cup \cT(\cH,X)^u$ the \emph{$(V\setminus X)$-projection} of $\cH$. 

\bigskip
\bigskip

We use \Cref{t-main-X-reduction} first to derive a sufficient condition for an efficient single step of the above variable elimination scheme.

\begin{definition}\label{def-nest-set}
	Given positive integers $\tau$ and $K$, let us call a nonempty subset $X\subseteq V$ a $(\tau,K)$-nest set of $\cH\subseteq 2^V$ if $|X|\leq \tau$ and $|\cT(\cH,X)^u|\leq K$.
\end{definition}

\begin{lemma}\label{lem-nest-set-computation}
Given a hypergraph $\cH\subseteq 2^V$, we can find all of its $(\tau,K)$-nest sets in $O(K|V|^{\tau +1}|\cH|)$ time. 
\end{lemma}

\begin{proof}
	We try all nonempty subsets $X\subseteq V$, $|X|\leq \tau$, and for each we compute $|\YY_{\cT(\cH,X)}|$. By \Cref{l-union-signature} we have $|\cT(\cH,X)^u|=|\YY_{\cT(\cH,X)}|$, and thus the claim follows by \Cref{t-signature-generation}. 
\end{proof}

\begin{corollary}\label{cor-one-step}
	Given a hypergraph $\cH\subseteq 2^V$ and two constants $\tau$ and $\gb$ we can test in $O(|V|^{\gb +\tau +1}|\cH|)$ time if it has a $(\tau,|V|^\gb)$-nest set, and if it has, then we can compute the projection \eqref{e-main-theorem} in $O(|V|^{\max(1+\gb,2\gb)}|\cH|)$ time. 
\end{corollary}

\begin{proof}
	Follows by \Cref{lem-nest-set-computation} and \Cref{t-main-X-reduction}.
\end{proof}

\bigskip
While such a single step variable elimination could be helpful in reducing/simplifying BPO problems, our purpose here is to solve BPO problems by recursively applying the above elimination step.  

We describe below a fully general solution algorithm for BPO problems that applies the above ideas. We will analyze its efficiency later  for special hypergraph classes.

\bigskip

\noindent\textsc{Procedure SBVE} (Signatures Based Variable Elimination)

\noindent\hspace*{1cm}\textbf{Input}: an integer threshold $\tau\in \ZZ_+$, \\
\noindent\hspace*{2.25cm}a hypergraph $\cH\subseteq 2^V$ and \\
\noindent\hspace*{2.25cm}a coefficient vector $\bc\in\ZZ^\cH$. 

\noindent\hspace*{1cm}\textbf{Initialize}: $k=1$, $W_1=V$, and $g_1=f_{\cH,\bc}$ (as in \eqref{e-BPO})

\noindent\hspace*{1cm}\textbf{While} $|W_k|\geq \tau$ \textbf{do}:

\noindent\hspace*{2.25cm}\textbf{Step 1}: Choose $X_k\subseteq W_k$, $1\leq |X_k|\leq \tau$, and set $Y_k=W_k\setminus X_k$.

\noindent\hspace*{2.25cm}\textbf{Step 2}: Compute $\cT_k=\cT(\cH(g_k),X_k)$ from $\cH(g_k)$ (as in \eqref{e-PQRT-updated})

\noindent\hspace*{2.25cm}\textbf{Step 3}: Compute $\YY_{\cT_k}$, 

\noindent\hspace*{2.25cm}\textbf{Step 4}: Compute $\mu(\by)$ and a binary vector $\bx(\by)\in \BB^{X_k}$ attaining \\$~$\hfill  the maximum in \eqref{e-mu}, for each $\by\in \YY_{\cT_k}$. 

\noindent\hspace*{2.25cm}\textbf{Step 5}: Compute $h=g_k^{Y_k}$ (as in \Cref{t-main-X-reduction}).

\noindent\hspace*{2.25cm}\textbf{Step 6}: Update $g_{k+1}=h$, $W_{k+1}=Y_k$, and set $k=k+1$.

\noindent\hspace*{1cm}\textbf{EndWhile}: 

\noindent\hspace*{1cm}\textbf{Finalize}: Compute $Z=\displaystyle\max_{\bx\in \BB^{W_k}} g_k(\bx)$, a binary vector $\bx_k$ attaining \\$~$\hfill this maximum, and set $\ell=k-1$. 

\noindent\hspace*{1cm}\textbf{While} $\ell \geq 1$ \textbf{do}:

\noindent\hspace*{2.25cm}\textbf{Step 7}: Compute $\by_\ell = \by(\bx_{\ell+1})\in \YY_{\cT_\ell}$.

\noindent\hspace*{2.25cm}\textbf{Step 8}: Set $\bx_\ell = (\bx(\by_\ell),\bx_{\ell+1})\in \BB^{X_\ell\cup Y_\ell}$ (concatenation)

\noindent\hspace*{2.25cm}\textbf{Step 9}: Update $\ell=\ell-1$.

\noindent\hspace*{1cm}\textbf{EndWhile}: 

\noindent\hspace*{1cm}\textbf{Output}: $Z$ and $\bx^*=\bx_1$.

\begin{corollary}\label{cor-SBVE}
	\textsc{Procedure SBVE} correctly outputs the optimum value $Z$ and an optimal solution $\bx^*$ of problem \eqref{e-BPO}. Furthermore, the logarithmic size of the coefficients of functions $h$ computed in Step 5 remain polynomially bounded by the input (binary) size and $\tau$ as long as $\tau$ is limited by a polynomial of the input size. 
\end{corollary}

\begin{proof}
	Follows by \Cref{t-main-X-reduction} and the definitions, \eqref{e-eliminateX} and \eqref{e-mu}.
\end{proof}

Let us remark last that Step 1 in the above procedure is quite vague (on purpose), while all other steps are described precisely, and are based on the computational ideas we used in some of our proofs above. We shall provide precise description for Step 1 tailored to the structure of the special hypergraph classes we consider below. In fact, we choose in Step 1 a $(\tau,K)$-nest for a particular value of $\tau=K$ such that the considered special hypergraph class guarantees not only the existence of such a nest set, but also its polynomial computability. Note that if $\tau$ is bounded polynomially by the input size, then the coefficients encountered in this procedure all stay polynomially bounded by the input size due to \Cref{t-main-X-reduction}.

\section{Applications}

In this section we utilize the explicit description of a projection in \Cref{t-main-X-reduction} allowing us to analyze the repeated application of this variable elimination (projection) process to describe specially structured families of hypergraphs for which \textsc{Procedure SBVE} can provide an efficient solution of the corresponding BPO problem.

\subsection{Union Bounded Problems}

First we consider those hypergraphs the union closure of which is ``not too large''. 

\begin{definition}\label{def-union-bounded}
	For an integer $K\in\ZZ_+$ we call a hypergraph $\cH$ \emph{$K$-union bounded} if $|\cH^u|\leq K$.
\end{definition}
For a $K$-union bounded hypergraph $\cH\subseteq 2^V$ we also call a corresponding multilinear polynomial $f_{\cH,\bc}$ and the corresponding BPO problem \eqref{e-BPO} \emph{$K$-union bounded} for every $\bc\in\ZZ^\cH$.

\begin{lemma}\label{l-union-bounded-recognition}
	$K$-union bounded hypergraphs can be recognized in $K\cdot O(|V||\cH|)$ time. 
\end{lemma}
\begin{proof}
	Follows by \Cref{t-signature-generation}.
\end{proof}
	
\begin{theorem}\label{t-main-union-bounded}
	Assume $\cH\subseteq 2^V$ is a hypergraph, and $V=X\cup Y$ is a nontrivial partition. Then for all $\bc\in\ZZ^\cH$ we have $|\cH\big(f^Y_{\cH,\bc}\big)^u|\leq |\cH^u|$ and $|\cT(\cH,X)^u|\leq |\cH^u|$. 
\end{theorem}

\begin{proof}
	By \Cref{t-main-X-reduction} and \Cref{l-union-signature} we have 
	\[
	\begin{array}{rl}
	\cH\big(f^Y_{\cH,\bc}\big)^u ~\subseteq ~ \big(\cR(\cH,X) \cup \cT(\cH,X)^u\big)^u &\subseteq ~ \big(\cR(\cH,X)\cup \cT(\cH,X)\big)^u\\
	&=~ \{U\setminus X\mid U\in\cH^u\}
	\end{array}
	\]
	from which both inequalities follow.
\end{proof}

\begin{corollary}\label{cor-union-bounded}
	The projection of a $K$-union bounded hypergraph $\cH$ is also $K$-union bounded. Furthermore, every subset $X\subseteq V$, $|X|\leq \tau$ is a $(\tau,K)$-nest set.  Consequently, for constant $\tau$ and $k$, an $O(n^k)$-union bounded BPO problem is solved by \textsc{Procedure SBVE} in strongly polynomial time. \qed
\end{corollary}

\begin{proof}
	The claim follows by \Cref{t-main-union-bounded}. Note that in Step 1 of \textsc{Procedure SBVE} we can choose an arbitrary subset of size at most $\tau$.
\end{proof}

Let us next argue that polynomially union bounded BPO problems can in fact be solved efficiently in one step, without a recursive variable elimination. 

For a subset $S\subseteq V$ we denote by $\gk(S)\in \BB^V$ the characteristic vector of $S$. To a hypergraph $\cH\subseteq 2^V$ we associate the family of characteristic vectors of the subsets in its union closure:

\begin{equation}\label{e-char-of-union-closure}
	\UU(\cH) ~=~ \{\gk(U)\mid U\in\cH^u\} ~\subseteq~ \BB^V.
\end{equation}

\begin{theorem}\label{t-BPO-reduction-to-union-closure}
	For all hypergraphs $\cH\subseteq 2^V$ and vectors $\bc\in \ZZ^\cH$ we have 
	\[
	\max_{\bx\in \BB^V} f_{\cH,\bc}(\bx) ~=~ \max_{\bx\in \UU(\cH)} f_{\cH,\bc}(\bx).
	\]
\end{theorem}
\begin{proof}
	To a binary vector $\bx\in\BB^V$ let us associate its on-set $ON(\bx)=\{v\in V\mid x_v=1\}$, and define 
	\[
	U(\bx)~=~ \bigcup_{H\in \cH\atop H\subseteq ON(\bx)} H. 
	\]
	Observe that for all $\bx\in\BB^V$ we have $U(\bx)\in\cH^u$ and that we have 
	\[
	f_{\cH,\bc}(\bx) ~=~ f_{\cH,\bc}(\gk(U(\bx)))
	\]
	from which the statement follows.
\end{proof}

\begin{corollary}\label{cor-union-bounded-one-step}
	The BPO problem with a $K$-union bounded hypergraph $\cH\subseteq 2^V$ can be solved in $K\cdot O(|V||\cH|)$ time. 
\end{corollary}

\begin{proof}
	According to \Cref{t-signature-generation} we can generate $\UU(\cH)$ in $K\cdot O(|V||\cH|)$ time, and then in the same time we can try each one of those binary vectors and choose the best. The claim follows by \Cref{t-BPO-reduction-to-union-closure}.
\end{proof}

\subsection{Locally Union Bounded Problems}

While union bounded hypergraphs are natural to consider in view of Theorem \ref{t-main-X-reduction}, they form a somewhat narrow class that has no relation to the other classes of hypergraphs for which BPO is known to be solvable efficiently. It is possible however to extend this class to a significantly larger family that still guarantees efficient solvability of problem BPO. 

Assume that $\tau\in\ZZ_+$ is a positive integer, and $K\in \ZZ_+$ as before. 

\begin{definition}\label{def-locally-bounded}
	We say that a hypergraph $\cH\subseteq 2^V$ is \emph{locally $(\tau,K)$-bounded} if
	\begin{itemize}
		\item[(i)] it has a $(\tau,K)$-nest, and 
		\item[(ii)] for each $(\tau,K)$-nest $X\subseteq V$ the projection hypergraph $\cT(\cH,X)^u\cup \cR(\cH,X)$ is also locally $(\tau,K)$-bounded. 
	\end{itemize}
\end{definition}

\begin{corollary}\label{cor-locally-bounded}
	Assume that $\tau$ and $\gb$ are given constants and $K=O(n^\gb)$. Then \textsc{Procedure SBVE} will solve a locally $(\tau,K)$-bounded BPO problem in strongly polynomial time, if in Step 1 of SBVE we choose a $(\tau,K)$-nest. 
\end{corollary}

\begin{proof}
	By property (ii) of \Cref{def-locally-bounded} the projection of a locally $(\tau,K)$-bounded BPO problem is also locally $(\tau,K)$-bounded. Thus we can proceed with \textsc{Procedure SBVE} with the restriction that in Step 1 in each iteration we choose a $(\tau,K)$-nest the existence of which is guaranteed by \Cref{def-locally-bounded} and our assumption that the input hypergraph $\cH$ is locally $(\tau,K)$-bounded. Then the claim follows by \Cref{t-main-X-reduction}, \Cref{cor-SBVE}, and \Cref{lem-nest-set-computation}.
\end{proof}

Let us note that we do not know how to test the conditions of \Cref{def-locally-bounded} in polynomial time. Furthermore, it may happen that we have different subsets $X_1\subseteq V$ and $X_2\subseteq V$, both satisfying property (i) of \Cref{def-locally-bounded}, and with $\cG=\cT(\cH,X_1)^u\cup\cR(\cH,X_1)$
\[
|\cT(\cG,X_2)^u| ~\gg~ K. 
\]
Still, under further special properties of the underlying hypergraph, local boundedness may become efficiently recognizable. 

\bigskip

In the next subsections we show multiple classes of hypergraphs that are $O(n^\gb)$-union bounded or locally $(\tau,O(n^\gb))$-bounded, for some constants $\tau$ and $\gb$. In fact these classes include all of the tractable special classes cited in the introduction, and some others.


\subsection{Bounded width hypergraphs}

Given a hypergraph $\cH\subseteq 2^V$ we call a subfamily $\cC\subseteq \cH$ a \emph{chain} if for all $H,H'\in\cC$ we have either $H\subseteq H'$ or $H'\subseteq H$ (such families are also called nested). A subfamily $\cA\subseteq \cH$ is called an \emph{antichain} if for all $H,H'\in \cA$, $H\neq H'$ we have both $H\setminus H'\neq \emptyset$ and $H'\setminus H\neq \emptyset$ (such families are also called Sperner systems in the literature). The celebrated theorem of Dilworth \cite{Dilworth-1950} claims that the minimum number of chains in a chain cover of $\cH$ is equal to the maximum number of hyperedges in an antichain of $\cH$, i.e., 
\[
\max 
\left\{|\cA| \left| 
\begin{array}{c}
	\cA\subseteq \cH \\
	 \cA \text{ is an antichain}
\end{array}
\right.\right\}
~=~ 
\min 
\left\{\ell \left| 
\begin{array}{c}
\exists \text{ chains } \cC_j\subseteq \cH, ~j\in [\ell] \\
\text{ s.t. } \cH=\cC_1\cup \dots \cup \cC_\ell
\end{array}
\right.\right\}.
\] 
Let us denote by $\go(\cH)$ this common value, called the \emph{width} of $\cH$. 

\begin{theorem}\label{t-antichain}
		Any hypergraph $\cH\subseteq 2^V$ is $(|V|+1)^{\go(\cH)}$-union bounded. Consequently, the BPO problem for a hypergraph with bounded width is solvable in strongly polynomial time by \textsc{Procedure SBVE} using $\tau=1$ or by the direct approach of \Cref{cor-union-bounded-one-step}.
\end{theorem}

\begin{proof}
	Let us observe first that $\emptyset\in\cH$ can be assumed w.l.o.g., since adding the empty set to a hypergraph is not changing the unions obtainable from that hypergraph. Let us next consider a chain cover $\cH=\cC_1\cup \dots \cup \cC_\ell$ with $\ell=\go(\cH)$. By the above cited theorem of Dilworth such a chain cover of $\cH$ exists. We can also assume w.l.o.g. that each chain $\cH_j$, $j\in [\ell]$ is a maximal chain within $\cH$. This is because extending a chain to a maximal one (if needed) does not to change the fact that these chains form a chain cover of $\cH$. 
	
	Let us next observe that to an arbitrary subfamily $\cF\subseteq \cH$ we can associate hyperedges
	$ H_j\in\cC_j$, $j\in [\ell]$ by defining $H_j$ as the maximal set in $\cF\cap \cC_j$ if this intersection is not empty, and setting $H_j=\emptyset$ if $\cF\cap \cC_j=\emptyset$. Note that since $\cC_j$ is a chain, $H_j$ is uniquely defined, for each $j\in [\ell]$, and furthermore we have the equality
	\[
	\bigcup_{H\in \cF} H ~=~ \bigcup_{j\in [\ell]} H_j.
	\]
	Since a maximal chain has at most $|V|+1$ sets, the right hand side above can define at most $(|V|+1)^\ell$ different sets, proving the claimed union boundedness
	\[
	|\cH^u| ~\leq~ (|V|+1)^\ell.
	\]
	Note that by \Cref{t-main-union-bounded} an arbitrary projection of a $(|V|+1)^\ell$-union bounded hypergraph is again $(|V|+1)^\ell$-union bounded. Thus we can run \textsc{Procedure SBVE} with $\tau=1$, eliminating variables in an arbitrary order. 
	
	Consequently, the claimed strongly polynomial performance of \textsc{Procedure SBVE} follows by \Cref{t-main-X-reduction} for any constants $\gk$ and hypergraphs $\cH\subseteq 2^V$ with $\go(\cH)\leq \gk$. 
\end{proof}

\subsection{Bounded tree-width hypergraphs}

In this subsection we consider simple undirected graphs and use standard graph theory notions and definitions, see e.g., \cite{Bondy-Murty-1976}. We say that a vertex $v$ of a graph $G=(V,E)$ is \emph{simplicial} if its neighbors $N(v)=\{w\in V\mid (v,w)\in E\}$ induce a complete subgraph of $G$. For $\gk\in \ZZ_+$ we say that for a graph $G=(V,E)$ an ordering $V=\{v_1,v_2,\dots ,v_n\}$ is a \emph{$\gk$-perfect elimination order} if for every $v_j\in V$ vertex $v_j$ is simplicial in the subgraph induced by $W_j=\{v_j,v_{j+1},\dots,v_n\}$ and $|N(v_j)\cap W_j|\leq \gk$. A graph $G$ is called a \emph{$\gk$-tree} if it has a $\gk$-perfect elimination order. A graph $G=(V,E)$ is called a \emph{partial $\gk$-tree} if there exists a $\gk$-tree $G^*=(V,F)$ such that $E\subseteq F$. The \emph{tree-width} $tw(G)$ of a graph $G$ is the smallest value of $\gk$ such that $G$ is a partial $\gk$-tree. While determining $tw(G)$ is NP-hard in general (see e.g., \cite{Arnborg-1987}), for a constant $\gk$ the relation $tw(G)\leq \gk$ can be recognized and if the relation holds the corresponding $\gk$-tree and $\gk$-perfect elimination order can be constructed in polynomial time, see e.g., \cite{Arnborg-1985,Bodlaender-1996,Robertson-1986}.

To a hypergraph $\cH\subseteq 2^V$ we associate its \emph{co-occurrence graph} $G_\cH=(V,E)$ where 
\[
E ~=~ \{(i,j)\mid i,j\in V,~ \exists H\in\cH \text{ s.t. } i,j\in H \}
\]
and define its \emph{tree-width} $tw(\cH)$ as the tree-width of its associated co-occurrence graph $tw(\cH)=tw(G_\cH)$.

\begin{theorem}\label{cor-limited-tree-width}
	For any constant $\gk\in \ZZ_+$ a hypergraph $\cH\subseteq 2^V$ with $tw(\cH)\leq \gk$ is locally $(1,2^\gk)$-bounded. Consequently, the BPO problem for hypergraphs with bounded tree-width is solvable in strongly polynomial time by \textsc{Procedure SBVE} using $\tau=1$.
\end{theorem}

\begin{proof}
	Let us first construct the co-occurrence graph $G_\cH=(V,E)$ and the $\gk$-tree $G^*=(V,F)$ for which $E\subseteq F$. Assume that we relabel the elements of $V$ according to the $\gk$-perfect elimination order of $G^*$ and that $v\in V$ is the first vertex in this order. Then, by the definition of the $\gk$-perfect elimination order, the neighborhood $N(v)$ of $v$ in $G^*$ is a clique with $|N(v)|\leq \gk$. 
	
	Let us then consider $X=\{v\}$ and the projection of $\cH$ on $V\setminus X$. Since $E\subseteq F$, the neighbors of $v$ in $G_\cH$ are all belong to $N(v)$. Consequently, by the definition of the co-occurrence graph we have $\cT(\cH,X)\subseteq 2^{N(v)}$, implying $\cT(\cH,X)^u\subseteq 2^{N(v)}$. These imply on the one hand that $|\cT(\cH,X)^u|\leq 2^\gk$ and on the other hand that for the co-occurrence graph $G_{\cT(\cH,X)^u\cup \cR(\cH,X)}=G'=(V\setminus X, E')$ we also have $E'\subseteq F$. Thus, $G'$ is also a partial $\gk$-tree with respect to the $\gk$-tree that is the subgraph of $G^*$ induced by $V\setminus X$.
	
	Consequently, we can execute \textsc{Procedure SBVE} with $\tau=1$ and choose $X_k=\{k\}$ in Step 1 of iteration $k$ guaranteeing by the above analysis that $|\YY_{\cT_k}|\leq 2^\gk$. Since $\gk$ is a constant, the procedure will provide a solution to this BPO problem  in strongly polynomial time according to \Cref{t-main-X-reduction} and \Cref{cor-SBVE}. 
\end{proof}

For a more precise complexity of the above result note that we can compute the co-occurrence graph of $\cH$ in $O(|V|^2|\cH|)$ time, and test $tw(G_\cH)\leq \gk$ in $O(2^{O(\gk^3)}|V|)=O(|V|)$ time by the algorithm of \cite{Bodlaender-1996}.  \textsc{Procedure SBVE} takes $|V|$ iterations, and in each iteration we have $|\cT_k^u|\leq 2^\gk$ implying by \Cref{t-main-X-reduction} that the procedure will terminate in $O(|V|^2|\cH|)$ time. 

\bigskip


This family of hypergraphs was recently considered by \cite{Clausen-Crama-Lusby-Heck-Ropke-2024}. For a hypergraph $\cH\subseteq 2^V$ we define its \emph{reach} as $\gr(\cH)=\max_{H\in\cH}\max _{i,j\in H} |i-j|$. The authors of \cite{Clausen-Crama-Lusby-Heck-Ropke-2024} introduced a specialized dynamic programming procedure for this special class, and showed strongly polynomial tractability of problem BPO for this case. Note that $tw(\cH)\leq \gr(\cH)$ as observed in \cite{Clausen-Crama-Lusby-Heck-Ropke-2024}. In fact, for $\gk=\gr(\cH)$ the graph $G^*=(V,F)$ with $F=\{(i,j)\mid i,j\in V,~|i-j|\leq \gk\}$ is a $\gk$-tree, such that for $G_\cH=(V,E)$ we have $E\subseteq F$, and $V=\{1,2,\dots,n\}$ is a $\gk$-perfect elimination order.
Consequently, the results of the previous subsection provide strongly polynomial time solution for problem BPO for hypergraphs with bounded reach, by eliminating variables from $V$ one by one, in the order indicated above. Thus, for hypergraphs with limited reach we do not need to spend time to compute the actual tree-width and the corresponding perfect elimination order. 

\bigskip

\subsection{$\gb$-acyclic hypergraphs}

Let us recall (see \cite{Fagin-1983b}) that a hypergraph $\cH\subseteq 2^V$ is $\beta$-acyclic, if it does not contain hyperedges $H_i\in\cH$ and vertices $v_i\in V$ for $i=1,...,k$ for some $k\geq 3$ such that $v_i\in H_{i-1}\cap H_i$ and $v_i\not\in H_j$ if $j\not\in \{i-1,i\}$ for all $i,j=1,...,k$ (where $1=k+1$). Such hypergraphs were called \emph{totally balanced} in \cite{Lovasz-1979} and have interesting properties, see e.g., \cite{Anstee-1983}.

A vertex $v\in V$ is called a \emph{nest point} of $\cH$ if $\cP(\cH,\{v\})\cup \cQ(\cH,\{v\})$ is a \emph{nested family}, that is if for any pair $H,H'\in \cP(\cH,\{v\})\cup \cQ(\cH,\{v\})$ we have either $H\subseteq H'$ or $H'\subseteq H$. Note that this condition is equivalent with the nestedness of $\cT(\cH,\{v\})$. 

It was shown in \cite{Brouwer-Kolen-1980} (see also \cite{Brault-Baron-2017,Duris-2012,Jegou-Ndiaye-2009}) that a hypergraphs $\cH\subseteq 2^V$ is $\gb$-acyclic if and only if for every subset $X\subseteq V$ the subhypergraph $\cT(\cH,X)\cup \cR(\cH,X)$ has a nest point. Since for a nested hypergraph $\cN$ we have both $\cN^u=\cN$ and $|\cN|\leq |V|$, the following claim can be shown easily.

\begin{theorem}\label{cor-beta-acyclic}
A $\gb$-acyclic hypergraph $\cH\subseteq 2^V$ is locally $(1,|V|)$-bounded. Consequently, the BPO problem for $\gb$-acyclic hypergraphs is solved in strongly polynomial time by \textsc{Procedure SBVE} using $\tau=1$.
\end{theorem}

\begin{proof}
	Note first that if we choose $X=\{v\}$ for a nest point $v\in V$ of $\cH$ then $\cT(\cH,X)$ is a nested family, implying $\cT(\cH,X)^u=\cT(\cH,X)$ and $|\cT(\cH,X)|\leq |V|$. Furthermore, by the characterization of \cite{Brouwer-Kolen-1980} the projected family $\cT(\cH,X)\cup \cR(\cH,X)$ is again $\gb$-acyclic. Thus \Cref{def-locally-bounded} holds with $(\tau,K)=(1,|V|)$. Furthermore, in iteration $k$ in Step 1 of \textsc{Procedure SBVE} we can choose the set $X_k=\{v\}$ for a nest point $v\in W_k$ of the $\gb$-acyclic hypergraph $\cH(g_k)$. Thus the claim follows by \Cref{t-main-X-reduction} and \Cref{cor-SBVE}.
\end{proof}

Note that the complexity results claimed in \Cref{t-main-X-reduction} are based on the general projection setting and yield an $O(|\cH||V|^3)$ complexity for $\gb$-acyclic hypergraphs.  However, in this special case we can compute the coefficients on the right hand side of \Cref{t-main-X-reduction} more efficiently, utilizing the nested structure of $\cT$, potentially improving total complexity to $O(|V|^4)$ after an initial one time $O(|V||\cH|)$ preprocessing.

\subsection{Bounded rank $\ga$-acyclic hypergraphs}

Let us recall (see \cite{Fagin-1983b}) that a hypergraph $\cH\subseteq 2^V$ is $\ga$-acyclic
if its hyperedges can be labeled $\cH=\{H_0,H_1,\dots,H_m\}$ such that for all $j=1,\dots m$ there exists an index $\ell=\ell(j)<j$ such that
\begin{equation}\label{e-running-intersection}
	H_j\cap \left(\bigcup_{i<j} H_i\right) ~\subseteq~ H_\ell.
\end{equation}
Note that we can assume w.l.o.g. that $H_0=\emptyset$, since we can always add the emptyset to an $\ga$-acyclic hypergraph without violating $\ga$-acyclicity. If for an index $j$ there exist more than one $\ell$ that satisfies \eqref{e-running-intersection}, then we denote by $\ell(j)$ the smallest such index. 

The above property is called the \emph{running intersection property} and it characterizes the large family of $\ga$-acyclic hypergraphs \cite{Beeri-Fagin-Maier-Yannakakis-1983}. For instance, any hypergraph becomes $\ga$-acyclic, if we add to it $V$ as a hyperedge. Consequently, the BPO problem remains NP-hard, when restricted to $\ga$-acyclic hypergraphs (see also \cite{DelPia-DiGregorio-2023}). Let us further add that $\ga$-acyclicity can be recognized in $O(\sum_{H\in\cH} |H|)$ time (see e.g., \cite{graham1979universal-1979,yu1979algorithm-1979}) and an ordering of the hyperedges of $\cH$ satisfying the running intersection property can also be computed in the same time (see \cite{maier1983theory}). 

The \emph{rank} of a hypergraph $\cH\subseteq 2^V$ is defined as 
\begin{equation}\label{e-rank-of-cH}
	r(\cH) ~=~ \max_{H\in \cH} |H|. 
\end{equation}
We say that $\cH$ is a \emph{bounded rank hypergraph} if $r(\cH)\leq d=O(\log poly (|V|,|E|))$.
Using an enhanced linearization method it was shown recently in \cite{DelPia-Khajavirad-2025} that BPO problems with bounded rank $\ga$-acyclic hypergraphs are polynomially solvable. 

We show here that such hypergraphs are also locally union bounded and \textsc{Procedure SBVE} can be shown to solve such BPO instances in strongly polynomial time. To be able to show these claims, we need a few technical observations. 

Assume for the rest of this subsection that $\cH=\{H_0=\emptyset, H_1,\dots , H_m\}\subseteq 2^V$ is an $\ga$-acyclic hypergraph satisfying the running intersection property \eqref{e-running-intersection}
and that $r(\cH)\leq d=O(\log poly (|V|,|E|))$. We set $J=\{0,1,\dots, m\}$ for this subsection.  Furthermore, we can assume w.l.o.g. that $V=\bigcup_{H\in\cH} H$, since otherwise we could just redefine $V$.

For an index $v\in V$ let us define 
\[
i(v) ~=~ \min_{H_j\ni v} j,
\]
set $I=\{i(v)\mid v\in V\}\subseteq J$, and define 
\[
S_i ~=~ \{v\in V\mid i(v)=i\} ~~\text{ for all }~~ i\in I.
\]
Note that by our assumptions, we have $1\leq i(v)\leq m$ for all $v\in V$. 

\begin{lemma}\label{lem-gaacyclic-Si}
	The sets $S_i$, $i\in I$ form a partition of $V$, and for all $i\in I$ we have $S_i\subseteq H_i$ and hence $|S_i|\leq d$. 
\end{lemma}

\begin{proof}
	Follows immediately by the above definitions and the assumption that $r(\cH)\leq d$.
\end{proof}

\begin{lemma}\label{lem-gaacyclic-unions}
	For all $k\in J$ we have 
		\[
		\bigcup_{j\in J\atop j<k}H_j ~=~ \bigcup_{i\in I\atop i<k} S_i.
		\]
\end{lemma}

\begin{proof}
	Let us denote the left hand side by $LH$ and the right hand side by $RH$.
	Observe first that for all $v\in RH$ we have $v\in S_i$ for some $i<k$ and thus by \Cref{lem-gaacyclic-Si} $v\in H_i$, proving $RH\subseteq LH$.
	Observe next that for all $v\in LH$ we have $v\in H_i$ for some $i<k$, implying $i(v)\leq i<k$ and thus $v\in S_{i(v)}\subseteq RH$, implying $LH\subseteq RH$.
\end{proof}

To every hyperedge $H\in\cH$ let us associate 
\[
i(H) ~=~ \max_{v\in H} i(v),
\]
and for every $i\in I$ define $J(i)=\{j\in [m]\mid i(H_j)=i\}$. 

\begin{lemma}\label{lem-gaacyclic-J(i)}
	For every $i\in I$ and $j\in J(i)$ we have $i\in J(i)$ and $i\leq j$. 
\end{lemma}

\begin{proof}
	By the definition of $i(v)$, $v\in S_i$, if $H_j\cap S_i\neq \emptyset$, then we must have $j\geq i$. Thus $H_i\cap S_{i'}=\emptyset$ must hold for all $i'\in I$, $i'>i$.
	Furthermore, by \Cref{lem-gaacyclic-Si} we have $S_i\subseteq H_i$, thus $i(H_i)=i$ also follows. 
\end{proof}

\begin{lemma}\label{lem-gaacyclic-J(i)-subseteq-Hi}
	For $i^* ~=~ \max I$ and $j\in J(i^*)$ we have $H_j\subseteq H_{i^*}$. 
\end{lemma}

\begin{proof}
	Assume for a contradiction that the claim is not true and choose a smallest counter example $j\in J(i^*)$. By \Cref{lem-gaacyclic-J(i)} we have $i^*\in J(i^*)$, and since $i^*$ is not a counterexample we must have $j>i^*$. 
	
	By the running intersection property there exists an $\ell(j)<j$ such that 
	\[
	H_j\cap \left(\bigcup_{j'\in J\atop j'<j} H_{j'}\right) ~\subseteq~ H_{\ell(j)}.
	\]
	Since $j>i^*$, by \Cref{lem-gaacyclic-unions} we have 
	\[
	\bigcup_{j'\in J\atop j'<j} H_{j'} ~=~ \bigcup_{i\in I\atop i<j} S_i ~=~ \bigcup_{i\in I} S_i ~=~ V
	\]
	and thus we have
	\begin{equation}\label{e-j-ellj}
		H_j\subseteq H_{\ell(j)}.
	\end{equation}
	This implies that $i(H_{\ell(j)})\geq i(H_j)=i^*$ from which $i(H_{\ell(j)})=i^*$ follows, since $i^*$ is the largest index in $I$. Consequently, $\ell(j)\in J(i^*)$. Since $\ell(j)<j$ it cannot be a counter example by our choice of $j$, hence $H_{\ell(j)}\subseteq H_{i^*}$. By \eqref{e-j-ellj} this contradicts the fact that $j$ is a counter example. This contradiction proves our claim. 
\end{proof}

\begin{lemma}\label{lem-gaacyclic-smallest-ell}
	For $i^* ~=~ \max I$ and $j > i^*$ we have $\ell(j)\leq i^*$.
\end{lemma}

\begin{proof}
	Let us recall first that $\ell(j)$ is the smallest $\ell$ that satisfies the running intersection property \eqref{e-running-intersection}. Assume for a contradiction that there exists a $j>i^*$ for which $\ell(j)>i^*$. Note that in this case by \Cref{lem-gaacyclic-unions} we have 
	\[
	\bigcup_{i<j} H_i ~=~ \bigcup_{i\in I} S_i ~=~ \bigcup_{i<\ell(j)} H_i.
	\]
	Consequently, by \eqref{e-running-intersection} we have the relations
	\[
	H_j ~\subseteq~ H_{\ell(j)} ~\subseteq~ H_{\ell(\ell(j))}
	\]
	contradicting that $\ell(j)$ is the smallest index satisfying \eqref{e-running-intersection} for $j$, since $\ell(\ell(j)) ~<~ \ell(j)$.
\end{proof}

\begin{lemma}\label{lem-gaacyclic-nestset}
	The set $X=S_{i^*}$, where $i^* ~=~ \max I$,  is a $(d,2^d)$-nest set of $\cH$. 
\end{lemma}

\begin{proof}
	By \Cref{lem-gaacyclic-Si} we have $|X|\leq d$. By our indexing the sets of $\cH$ we have
	$\cT(\cH,X)=\{H_j\setminus X\mid j\in J(i^*)\}$, and by \Cref{lem-gaacyclic-J(i)-subseteq-Hi} we have $H_j\setminus X\subseteq H_{i^*}\setminus X$ for all $j\in J(i^*)$. Since $|H_{i^*}\setminus X|\leq |H_{i^*}|\leq r(\cH)\leq d$, all of the set of $\cT(\cH,X)$ are subsets of a set not larger than $d$. Hence $|\cT(\cH,X)^u|\leq 2^d$ follows. 
\end{proof}

\begin{lemma}\label{lem-gaacyclic-projection}
	The projection hypergraph $\cR(\cH,X)\cup \cT(\cH,X)^u$, where
	$X=S_{i^*}$ with $i^* ~=~ \max I$,  is $\ga$-acyclic and has rank $\leq d$.
\end{lemma}

\begin{proof}
	By \Cref{lem-gaacyclic-J(i)-subseteq-Hi} we have $\cT(\cH,X)\subseteq 2^{H_{i^*}\setminus X}$ and thus $\cT(\cH,X)^u\subseteq 2^{H_{i^*}\setminus X}$. Consequently, since $\cR(\cH,X)\subseteq \cH$ by its definition, the rank of the projection hypergraph is limited by the rank of $\cH$. 
	
	Let us now define $F_j=H_j$ for $j=0,1,\dots,i^*-1$, $F_{i^*}=H_{i^*}\setminus X$, and assume that the rest of $\cR(\cH,X)\cup \cT(\cH,X)^u$ is $\{F_{i^*+1},\dots, F_q\}$ for some integer $q$.  
	
	We claim that $\cF$ satisfies the running intersection property. This is clearly true for indices $j<i^*$, since we did not change those sets. For $j=i^*$ we have 
	\[
	F_{i^*}\cap \left(\bigcup_{i<i^*}F_i \right) ~\subseteq~ H_{i^*}\cap \left(\bigcup_{i<i^*}H_i \right) ~\subseteq~ H_{\ell(i^*)}=F_{\ell(i^*)}
	\]
	since $\ell(i^*)<i^*$ and we have $F_i=H_i$ for $i<i^*$. Finally, for $j>i^*$ we consider two cases.
	
	If $F_j=H_{j'}\setminus X$ for some $j'>i^*$, $j'\in J(i)$, then 
	we can write by \Cref{lem-gaacyclic-J(i)-subseteq-Hi} that
	\[
	F_j\cap \left(\bigcup_{i<j}F_i \right) ~\subseteq~ H_{i^*}\cap \left(\bigcup_{i< i^*}H_i \right) ~\subseteq~ H_{\ell(i^*)}=F_{\ell(i^*)}
	\]
	since we have $F_j=H_{j'}\subseteq H_{i^*}\setminus X$, and $\ell(i^*)<i^*$, implying the last equality. Thus, defining $\ell(j)=\ell(i^*)$ for such indices will work.
	
	If $F_j=H_{j'}$ for some $j'>i^*$, $j'\not\in J(i)$, then $F_j\cap X=\emptyset$ and $\ell(j)=\ell(j')$  will ensure \eqref{e-running-intersection}.
	
	Finally, if $F_j\in \cT(\cH,X)^u\setminus \cT(\cH,X)$, then by \Cref{lem-gaacyclic-J(i)-subseteq-Hi} we have $F_j\subseteq H_{i^*}\setminus X$, and $\ell(j)=\ell(i^*)$ will ensure \eqref{e-running-intersection}.
\end{proof}

\begin{corollary}\label{cor-gaacyclic}
	An $\ga$-acyclic hypergraph $\cH\subseteq 2^V$ of rank $r(\cH)\leq d$ is locally $(d,2^d)$-bounded. 
\end{corollary}

\begin{proof}
	By \Cref{lem-gaacyclic-nestset} $\cH$ has a $(d,2^d)$-nest set, and by \Cref{lem-gaacyclic-projection}, its projection, after the elimination of that nest set, is again $\ga$-acyclic. Thus, \Cref{def-locally-bounded} works with $\tau=d$ and $K=2^d$. 
\end{proof}

\begin{theorem}\label{t-gaacyclic}
	A BPO problem with an $\ga$-acyclic hypergraph $\cH\subseteq 2^V$ of rank $r(\cH)=O(\log poly(|V|,|\cH|))$ can be solved by \textsc{Procedure SBVE} in strongly polynomial time.
\end{theorem}

\begin{proof}
	Let us note first that given a hypergraph $\cH\subseteq 2^V$, we can recognize, as noted above, in polynomial time if it is $\ga$-acyclic, and if it is, a labeling of its hyperedges satisfying the running intersection property can also be constructed at the same time, see \cite{graham1979universal-1979,yu1979algorithm-1979,Beeri-Fagin-Maier-Yannakakis-1983,maier1983theory,Fagin-1983b}. Thus the claim follows by \Cref{cor-gaacyclic} and \Cref{t-main-X-reduction}. Note that our inductive argument here do not recourse by the projections of the BPO objective functions, but rather on the projection hypergraphs. In other words, as we progress, we consider BPO problems in which some terms may have zero coefficients.  Note that the projection hypergraphs are also $\ga$-acyclic by Lemma \ref{lem-gaacyclic-projection}, and we choose in Step 1 of \textsc{Procedure SBVE} the set $S_{i^*}$ as a nest set, according to Lemma \ref{lem-gaacyclic-nestset}.
	Note also that since $r(\cH)$ is not supposed to be a constant, we could not rely on Corollary \ref{cor-locally-bounded} to find an appropriate nest set in each iteration. \end{proof}

\subsection{Circular interval hypergraphs}

For a finite set $V$ let us denote by $C_V=(V,E)$ a directed graph on vertices $V$ that is a directed (Hamiltonian) cycle going through all vertices in $V$. Given such a cycle $C_V=(V,E)$ and a vertex $v\in V$ we denote by $v^-$ and $v^+$ the (unique) vertices for which we have $(v^-,v),~ (v,v^+) ~\in E$. We say that a subset $I\subseteq V$ is a \emph{circular interval} in $C_V$ if the subgraph of $C_V$ induced by $I$ is a directed path or $I=V$. We call a hypergraph $\cH\subseteq 2^V$ \emph{circular interval} with respect to $C_V$ if all hyperedges $H\in \cH$ are circular intervals in $C_V$. 

Note that a circular interval hypergraph may not be union bounded (e.g., if they contain many pairwise disjoint small intervals), may not have bounded tree-width (e.g., if they have long circular intervals), and may not be $\gb$-acyclic (e.g., if they have a chain of circularly overlapping circular interval hyperedges that cover $V$). 

Circular interval hypergraphs were considered in the literature under different names (circular-arc hypergraphs, circular consecutive ones property, etc.) see e.g., \cite{Kobler-Kuhnert-Verbitsky-2017,Quilliot-1984}. It was observed in \cite{Quilliot-1984} that circular interval hypergraphs can be recognized (and the corresponding directed cycle $C_V$ can be constructed) in polynomial time, using the results of \cite{Booth-Lueker-1976} together with a transformation introduced by \cite{Tucker-1972}. 

It was also observed that some hard optimization problems become tractable under the circular consecutive ones property, see e.g., \cite{Hochbaum-Levin-2006,Bartholdi-Orlin-Ratliff-1980}. In this subsection we would like to add problem BPO to this list. 

Let us first make a few easy, technical observations about the structure of circular interval hypergraphs. 

\begin{lemma}\label{l-circ-int-1}
	Assume $\cH\subseteq 2^V$ is a circular interval hypergraph with respect to the directed (Hamiltonian) cycle $C_V$. Then we have $|\cH|\leq |V|^2$. 
\end{lemma}

\begin{proof}
	Each hyperedge $H\in \cH$ is defined uniquely by its first and last element in the circular order of $C_V$. Thus the claim follows.  
\end{proof}

\begin{lemma}\label{l-circ-int-2}
	Assume $\cH\subseteq 2^V$ is a circular interval hypergraph with respect to the directed (Hamiltonian) cycle $C_V$, $v\in V$, and $\cF\subseteq \cH$ such that $F\cap\{v,v^+\}\neq \emptyset$ for all $F\in\cF$. Then $\cF^u$ is also circular interval with respect to $C_V$.
\end{lemma}

\begin{proof}
	Note first that if $I,J\in \cF$, $I\cap J\neq \emptyset$, then $I\cup J$ is a circular interval set (with respect to $C_V$). Note next that if $I,J\in \cF$, $I\cap \{v,v^+\}=\{v\}$, and $J\cap \{v,v^+\}=\{v^+\}$ then still the set $I\cup J$ is a circular interval (with respect to $C_V$). Thus the claim follows.
\end{proof}

Given a directed (Hamiltonian) cycle $C_V=(V,E)$ and vertex $v\in V$ we define its projection on $V\setminus\{v\}$ by $C^{-v}_V=(V\setminus\{v\}, (E\setminus \{(v^-,v),(v,v^+)\}) \cup \{(v^-,v^+)\})$. Note that $C^{-v}_V$ is again a directed (Hamiltonian) cycle on the set $V\setminus \{v\}$, in which elements $v^-$ and $v^+$ are consecutive.

\begin{theorem}\label{t-circularly-interval}
	A circular interval hypergraph $\cH\subseteq 2^V$ is locally $(1,|V|^2)$-bounded and any vertex is a $(1,|V|^2)$-nest. Consequently, the BPO problem for circular interval hypergraphs is solved in strongly polynomial time by \textsc{Procedure SBVE} using $\tau=1$.
\end{theorem}

\begin{proof}
	In fact we can run \textsc{Procedure SBVE} and eliminate variables $x_v$, $v\in V$ one by one, in an arbitrary order in this case. To see this assume that $\cH$ is a circular interval hypergraph with respect to the directed (Hamiltonian) cycle $C_V=(V,E)$, $v\in V$ is an arbitrary element, and $X=\{v\}$. Then the hyperedges in $\cQ(\cH,X)$ are all circular intervals of $C_V$ containing element $v$. Consequently, $\cT(\cH,X)$ is circular interval with respect to the projection $C^{-v}_V$, too. Furthermore, we have $T\cap \{v^-,v^+\}\neq \emptyset$ for all $T\in \cT(\cH,X)$. This implies by \Cref{l-circ-int-2} that the family $\cT(\cH,X)^u$ is also circular interval with respect to $C^{-v}_V$, since elements $v^-$ and $v^+$ are consecutive in $C^{-v}_V$. Thus, the claim follows by \Cref{l-circ-int-1}.
\end{proof}

Note that by applying the complexity bound of \Cref{t-main-X-reduction} we obtain an $O(|V|^5|\cH|)$ complexity for the BPO problem with circular interval hypergraph structure. 

Note also that we assume that the directed (Hamiltonian) cycle $C_V$ is also given, proving that $\cH$ is circular interval. This is similar to the case of hypergraphs with bounded reach, where a particular order of $V$ is also given and $\gr(\cH)$ is computed with respect that order. Finding an order of $V$ that minimizes the reach of the given hypergraph $\cH\subseteq 2^V$ is however an NP-hard optimization problem. 
As we mentioned earlier, we do not face such difficulties with the circular interval property. Given a hypergraph $\cH\subseteq 2^V$ we can recognize if it is circular interval and construct a corresponding directed (Hamiltonian) cycle $C_V=(V,E)$ in polynomial time in $|V|$ and $|\cH|$, see \cite{Booth-Lueker-1976,Tucker-1972}.

\bibliographystyle{plain}

\bibliography{PolynomialBinary}

\end{document}